\documentclass[11pt]{article}

\usepackage{amsfonts}
\usepackage{amsmath}

\newcommand{\beq}{\begin{equation}}
\newcommand{\eeq}{\end{equation}}
\newcommand{\bea}{\begin{eqnarray}}
\newcommand{\eea}{\end{eqnarray}}
\newcommand{\beas}{\begin{eqnarray*}}
\newcommand{\eeas}{\end{eqnarray*}}
\newcommand{\para}{\mathbin{\!/\mkern-5mu/\!}}

\newtheorem{theorem}{Theorem}[section]

\newtheorem{definition}[theorem]{Definition}
\newtheorem{proposition}[theorem]{Proposition}

\newtheorem{corollary}[theorem]{Corollary}
\newtheorem{lemma}[theorem]{Lemma}
\newtheorem{remark}[theorem]{Remark}
\newtheorem{example}[theorem]{Example}
\newtheorem{examples}[theorem]{Examples}
\newtheorem{foo}[theorem]{Remarks}

\newenvironment{proof}{\addvspace{\medskipamount}\par\noindent{\it
Proof}.}
{\unskip\nobreak\hfill$\Box$\par\addvspace{\medskipamount}}

\newcommand{\bG}{\mathbb G}

\newcommand{\ee}{\ell}

\newcommand{\bM}{\mathbb M}

\newcommand{\di}{\mathfrak h}
\newcommand{\M}{\mathbb M}

\newcommand{\R}{\mathbb R}

\title{Stochastic analysis on sub-Riemannian manifolds with transverse symmetries}
\author{Fabrice Baudoin}

\date{Department of Mathematics, Purdue University \\
 West Lafayette, IN, USA}

\begin{document}
\maketitle

\begin{center} 
\textit{Dedicated, with admiration, to Donald Burkholder}
\end{center}

\

\begin{abstract}
 We prove a geometrically meaningful stochastic representation of the derivative of the heat semigroup on sub-Riemannian manifolds with tranverse symmetries. This representation is obtained from the study of Bochner-Weitzenb\"ock type formulas for sub-Laplacians on 1-forms. As a consequence, we prove new hypoelliptic  heat semigroup gradient bounds under natural global geometric conditions. The results are new even in the case of the Heisenberg group which is the simplest example of a sub-Riemannian manifold with transverse symmetries.
 \end{abstract}

\tableofcontents

\section{Introduction}

As shown in the  monographs by Hsu \cite{Hsu} ,  Stroock \cite{stroock},  and Wang \cite{wang} stochastic analysis provides a set of powerful tools to study the geometry of manifolds. However, as of today, most of the applications are restricted to Riemannian geometry.  The goal of the present work is to introduce some  stochastic analysis tools in sub-Riemannian geometry. We will, in particular, focus on the special class of sub-Riemannian manifolds with transverse symmetries that was introduced in \cite{BG}.

\

A sub-Riemannian manifold is a smooth  manifold $\bM$
equipped with  a non-holonomic, or bracket generating, subbundle
$\mathcal H \subset T\bM$ and a fiber inner product $g_\mathcal{H}$. This means that if we denote by
$L(\mathcal H)$ the Lie algebra of the vector fields generated by
the global $C^\infty$ sections of $\mathcal H$, then $\text{span}
\{X(x)\mid X\in L(\mathcal H)\} = T_x(\bM)$ for every $x\in \bM$. 
We note that when $\mathcal
H = T\bM$,  a sub-Riemannian manifold is simply a Riemannian one and
thus sub-Riemannian manifolds encompass Riemannian ones. However,
some aspects of the geometry of  sub-Riemannian manifolds are
considerably less regular than their Riemannian ancestors. Some of
the major differences between the two geometries are the following:
\begin{enumerate}
\item The Hausdorff dimension  is usually greater than the manifold
dimension;
\item The sub-Riemannian distance to a point $x$  is in general not smooth on any pointed neighborhood of $x$;
\item The exponential map defined by the geodesics  is in general not a local diffeomorphism in a
neighborhood of the point at which it is based (see \cite{Rayner});
\item The space of horizontal paths joining
two fixed points may have singularities (the so-called abnormal geodesics, see \cite{Montgomery}).
\item The sub-Riemannian Brownian motion $(X_t)_{t \ge 0}$ does not fill the space in an isotropic way as the Riemannian Brownian does. Intuitively, for small times $t$ the process $X_t$ will move at a speed $\sqrt{t}$ in the direction of the vector fields in $\mathcal{H}$,  $t$ in the direction of the vector fields in $[\mathcal{H},\mathcal{H}]$, $t^{3/2}$ in the direction of the vector fields in $[[\mathcal{H}, [\mathcal{H},\mathcal{H}]]$, and so on (see \cite{baudoin}).  
\item The sub-Laplacian is only subelliptic and not elliptic, i.e. the diffusion matrix at a point $x$ is in general  not invertible at $x$.
\end{enumerate}

\

Sub-Riemannian geometry takes its roots in very old problems related to isoperimetry but was internationally brought to the attention of mathematicians  by E. Cartan's pioneering address \cite {cartan} at the Bologna International Congress of Mathematicians in 1928. Since then, it has been the focus of numerous studies by geometers. In particular, one should consult the monographs by Agrachev \cite{A}, Bella\"iche \cite{bellaiche}, Gromov \cite{Gromov2} and Montgomery \cite{Montgomery} and the references therein. For the last four decades, sub-Riemannian geometry has also been a center of interest for analysts because it is the natural geometry associated to subelliptic partial differential equations (see \cite{RS,SC}). Perhaps more surprisingly, sub-Riemannian geometry has also been widely studied by probabilists since the breakthrough \cite{malliavin} by Malliavin in 1976, where stochastic analysis and hypoellipticity theory merged together. The paper gave birth to the nowadays called Malliavin calculus, which has then be successfully applied in the study of hypoelliptic heat kernels. We mention in particular the works by Ben Arous \cite{benarous}, and Kusuoka-Stroock \cite{KS1}. One may also consult the monograph \cite{baudoin} for further connections between probability theory and sub-Riemannian geometry.

Despite being an object of intensive studies,  partly due to the above obstructions most of the developments in sub-Riemannian geometry  to
date are of a local nature, that is are restricted to  compact manifolds only. As a consequence, the theory presently
lacks a body of results which, similarly to the  case of non compact manifolds ,
connect global properties of solutions of the relevant partial differential equations, or of the relevant stochastic processes to curvature properties of the ambient manifold.  

\

However, in some special sub-Riemannian structures, a notion of \textit{Ricci lower bound} has been made precise in several recent works \cite{BB,BBG,BG}. Numerous new hypoelliptic functional inequalities were then obtained as a consequence. We mention in particular the subelliptic Li-Yau inequalities (see \cite{BG}), the subelliptic parabolic Harnack inequalities (see \cite{BG}), the Poincar\'e inequalities on balls (see \cite{BBG}) and the log-Sobolev inequalities (see \cite{BB}).

\

In the present paper, by using probabilistic methods, we reprove and actually greatly improve under weaker conditions several inequalities that were obtained in \cite{BB} by using purely analytic methods. We also get new hypoelliptic inequalities which seem difficult to  prove directly by analysis. We mention, that in our opinion, the probability method is in a sense more direct and overall simpler than the analytic methods that were developed in \cite{BG}. More precisely, the results in Section 3 and 4 in \cite{BG} which were used to prove Hypothesis 1.4 in \cite{BG}  may now be omitted, since  this Hypothesis 1.4 is a straightforward consequence of Corollary \ref{GG} that we prove in this paper.

\

We now describe our main results. The paper is divided into two parts, a geometric part and a probabilistic part. The geometric part is devoted to the study of Bochner-Weitzenb\"ock type formulas on sub-Riemannian manifolds with transverse symmetries (see Section 2 for the definitions). More precisely, our goal will be to introduce a natural family $\square_\varepsilon$, $\varepsilon >0$, of sub-Laplacians on one-forms that satisfy the intertwining
\begin{equation}\label{commut}
dL=\square_\varepsilon d,
\end{equation}
where $L$ is the sub-Laplacian and $d$ the exterior derivative. The operator $\square_\varepsilon$ is  self-adjoint with respect to a Riemannian metric extension that contracts in the sense of Strichartz \cite{Strichartz} to the sub-Riemannian metric when $\varepsilon \to \infty$.  Our main geometric result is then Theorem \ref{sum} where we prove that
 \[
\square_\varepsilon=-(\nabla_\mathcal{H} -\mathfrak{T}_\mathcal{H}^\varepsilon)^* (\nabla_\mathcal{H} -\mathfrak{T}_\mathcal{H}^\varepsilon)+\frac{1}{2 \varepsilon}J^* J - \mathfrak{Ric}_{\mathcal{H}},
\]
and that for any smooth one-form $\eta$,
\[
\frac{1}{2} L \| \eta \|_{2\varepsilon}^2 -\langle \square_\varepsilon \eta , \eta \rangle_{2\varepsilon}=\sum_{i=1}^d  \| \nabla_{X_i} \eta  -\mathfrak{T}^\varepsilon_{X_i} \eta \|_{2\varepsilon}^2 +\left\langle \left(\mathfrak{Ric}_{\mathcal{H}}-\frac{1}{2 \varepsilon} J^* J\right)\eta, \eta \right\rangle_{2\varepsilon}.
\]
The quantities $\mathfrak{T}^\varepsilon$, $J^* J$ and $\mathfrak{Ric}_{\mathcal{H}}$ are tensors that will be introduced in the text. We should mention that, to our knowledge, this Bochner-Weitzenb\"ock formula is new even in the case of the Heisenberg group and it implies in a straightforward way the horizontal and the vertical Bochner's identities proved in \cite{BG}.

In the second part of the paper, we exploit the commutation \eqref{commut} to give a probabilistic representation of the derivative $dP_t$ where $P_t$ is the semigroup generated by the sub-Laplacian $L$. The representation actually follows from \eqref{commut} by adapting in our case classical ideas by Bismut \cite{bismut} Driver-Thalmaier \cite{Driver}, Elworthy \cite{Elworthy,Elworthy2} and Thalmaier \cite{thalmaier}. We deduce from this representation an integration by part formula in the spirit of Driver \cite{driver}. Several new hypoelliptic heat semigroup gradient estimates are then obtained as a  consequence. 

\

We point out that these two parts are largely independent and that the more probabilist reader may skip Section 3 since the main results proved in this section are summarized in Proposition \ref{gft}.

\

To conclude, we should mention that the inequalities we obtain, in the spirit of \cite{BB}, involve a vertical gradient. This, of course, does not mean that they are not geometrically meaningful, because we can see for instance  that the gradient bound in  Corollary \ref{GG} is actually equivalent to a global lower bound on the horizontal Ricci tensor of the sub-Riemannian connection. These hypoelliptic inequalities with vertical gradient have also been successfully been used in geometry, where real geometric theorems were proved as a consequence,  like the subelliptic Bonnet-Myers  \cite{BG} and in analysis where convergence to equilibrium for hypoelliptic kinetic Fokker-Planck equations were established \cite{baudoin2,villani}. On the other hand, we have to say that  the Driver-Melcher inequality \cite{Melcher} (see also \cite{BBBC,HQ} ) in the Heisenberg group, which only involves the horizontal gradient  still remains a little mysterious for us, and that it would of course be extremely interesting to connect those type of inequalities to natural geometric quantities.

\section{Sub-Riemannian manifolds with transverse symmetries}

The notion of sub-Riemannian manifold with transverse symmetries was introduced in \cite{BG}. We recall here the main geometric quantities and operators related to this structure that will be needed in the sequel and   we refer to \cite{BG} for further details. We also introduce some new geometric invariants that shall later be needed.

\

Let $\M$ be a smooth, connected  manifold with dimension $d+\di$. We assume that $\bM$ is equipped with a bracket generating distribution $\mathcal{H}$ of dimension $d$ and a fiberwise inner product $g_\mathcal{H}$ on that distribution. The distribution $\mathcal{H}$ is referred to as the set of \emph{horizontal directions}.  Sub-Riemannian geometry is the study of the geometry which is intrinsically associated to $(\mathcal{H},g_{\mathcal{H}})$ (see \cite{Strichartz}). In general, there is no canonical vertical complement of $\mathcal{H}$  in the tangent bundle $T\M$, but in some cases the fiberwise inner product $g_{\mathcal{H}}$ determines one.

\begin{definition}
It is said that $\M$ is a sub-Riemannian manifold with transverse symmetries if there exists a $\di$- dimensional Lie algebra $\mathcal{V}$ of sub-Riemannian Killing vector fields such that for every $x \in \bM$, 
 \[
 T_x \bM= \mathcal{H}(x) \oplus \mathcal{V}(x),
 \]
 where 
 \[
  \mathcal{V}(x)=\{ Z(x), Z \in \mathcal{V}(x) \}.
 \]
\end{definition}
We recall that a vector field $Z$ on $\M$ is called a sub-Riemannian Killing field if:
\begin{itemize}
\item The flow generated by $Z$ infinitesimally preserves $\mathcal{H}$, that is for every horizontal vector field $X$ (that is a smooth section of $\mathcal{H}$), the vector field $[Z,X]$ is horizontal.
\item The flow generated by $Z$ infinitesimally preserves the metric $g_{\mathcal{H}}$, that is $\mathcal{L}_Z g_\mathcal{H} =0$, where $\mathcal{L}_Z$ denotes the Lie derivative in the direction of $Z$.
\end{itemize}
Some of the most interesting   examples of sub-Riemannian manifolds with transverse symmetries come  from a principal fiber bundle projection  $\pi: \M \to \mathbb{N}$ with totally geodesic fibers isomorphic to the structure group. The sub-Riemannian  objects of $\M$ we are interested in  are then the lifts of the Riemannian objects of $\mathbb{N}$: The sub-Laplacian on $\mathbb{M}$ is the lift of the Laplace-Beltrami operator on $\mathbb{N}$ and  the horizontal Brownian motion on $\mathbb{M}$ which is our main object of interest is the lift of the Brownian motion on $\mathbb{N}$. The study of the horizontal  diffusion processes associated to this type  of submersions has already attracted a lot of attention in the past, mostly in connection with skew-product type decomposition theorems (see Elworthy-Kendall  \cite{Elworthy4} or Liao \cite{Liao}). In our work, we are more interested in developing an intrinsic horizontal stochastic calculus rather than skew-product considerations. Though a sub-Riemannian manifold with transverse symmetries may not be globally associated with a submersion, it is always locally.  More precisely, as recently observed by Elworthy in   \cite{Elworthy3}, a sub-Riemannian structure with transverse symmetries induces on $\mathbb{M}$ a Riemannian foliation with totally geodesic leaves and bundle-like metric.  We refer to the monograph by Elworthy-Le Jan-Li \cite{bookEl} for a discussion of diffusions on foliated manifolds.

\

From now on in the sequel of the paper, we assume that $\M$ is a sub-Riemannian manifold with transverse symmetries.

 The distribution $\mathcal{V}$ is referred  to as the set of \emph{vertical directions}.  The choice of an inner product $g_{\mathcal{V}}$ on the Lie algebra $\mathcal{V}$ naturally endows $\bM$ with a one-parameter family of Riemannian metrics that makes the decomposition $\mathcal{H} \oplus \mathcal{V}$ orthogonal:
\[
g_{\varepsilon}=g_\mathcal{H} \oplus  \frac{1}{\varepsilon }g_{\mathcal{V}}, \quad \varepsilon >0.
\]
 For notational convenience, we will often use the notation $\langle \cdot, \cdot \rangle_\varepsilon$, resp. $\langle \cdot ,\cdot \rangle_\mathcal{H}$,  resp $\langle \cdot ,\cdot \rangle_\mathcal{V}$, instead of $g_\varepsilon$, resp. $g_\mathcal{H}$, resp. $g_\mathcal{V}$.  We can extend $g_\mathcal{H}$ on $T_x\M \times T_x \M$ by the requirement that $g_\mathcal{H}(u,v)=0$ whenever $u$ or $v$ is in $\mathcal{V}(x)$. We similarly extend $g_\mathcal{V}$. Hence for any $u \in T_x\M$,
 \[
 \| u \|_{\varepsilon}^2=\| u\|_\mathcal{H}^2 +\frac{1}{\varepsilon} \| u\|_\mathcal{V}^2.
 \]
 
 Although  $g_\varepsilon$ will be useful for the purpose of computations, the geometric objects that we are eventually interested in, like the sub-Laplacian $L$ and its associated semigroup will of course not depend on $\varepsilon$.   

The Riemannian volume measure of $(\bM,g_\varepsilon)$ is always a multiple of the Riemannian volume measure of $(\bM,g_1)$, therefore we will always use the Riemannian volume measure of $(\bM,g_1)$ which we will denote $\mu$.

At every point $x\in \bM$, we can find a local  frame of vector fields $\{X_1,\cdots,X_d, Z_1, \cdots, Z_\di\}$ such that on a neighborhood of $x$: 
 \begin{itemize}
\item[(a)] $\{X_1,\cdots,X_d \}$ is a $g_\mathcal{H}$-orthonormal basis of $\mathcal{H}$;
\item[(b)] $\{Z_1, \cdots, Z_\di \}$ is a $g_\mathcal{V}$-orthonormal basis of $\mathcal{V}$.
\end{itemize}

We observe that  the following commutation relations hold:
\begin{equation}\label{bra1}
[X_i,X_j]=\sum_{\ee=1}^d \omega_{ij}^\ee X_\ee +\sum_{m=1}^{\di}
\gamma_{ij}^{m} Z_{m},
\end{equation}
\begin{equation}\label{bra2} 
[X_i,Z_{m}]=\sum_{\ee=1}^d \delta_{im}^\ee X_\ee,
\end{equation}
for 
smooth functions $ \omega_{ij}^\ee $, $ \gamma_{ij}^{m}$ and
$\delta_{im}^\ee $ such that
\begin{equation}\label{deltas}
\delta_{im}^\ee=-\delta_{\ee m}^i,\ \ i, \ee = 1,...,d,\
\text{and}\ m=1,...,\di.
\end{equation}
Property \eqref{deltas} follows from the property of $Z_m$ being a sub-Riemannian Killing field.
By convention,
$\omega_{ij}^\ee =-\omega_{ji}^\ee$, $\gamma_{ij}^{m}=-\gamma_{ji}^{m}$ and $\delta_{im}^\ee=-\delta_{mi}^\ee$.

We define the horizontal gradient $\nabla_\mathcal{H} f$ of a function $f$ as the projection of the Riemannian gradient of $f$ on the horizontal bundle. Similarly, we define the vertical gradient $\nabla_\mathcal{V} f$ of a function $f$ as the projection of the Riemannian gradient of $f$ on the vertical bundle.
In a local adapted frame, we have
\[
\nabla_\mathcal{H} f=\sum_{i=1}^d (X_i f) X_i,
\]
and
\[
\nabla_\mathcal{V} f =\sum_{m=1}^\di (Z_m f) Z_m.
\]
The canonical sub-Laplacian in a sub-Riemannian manifold with transverse symmetries is the generator of the symmetric Dirichlet form
\[
\mathcal{E}_{\mathcal{H}} (f,g) =\int_\bM \langle \nabla_\mathcal{H} f , \nabla_\mathcal{H} g \rangle_{\mathcal{H}} d\mu.
\]
It is a diffusion operator $L$ on $\bM$ which is symmetric  on $C^\infty_0 (\bM)$ with respect  to the measure $\mu$.

Actually, it is readily seen that in an adapted frame, one has
\begin{equation*}
L= -\sum_{i=1}^d X_i^*X_i,
\end{equation*}
where $X_i^*$ is the formal adjoint of $X_i$. From the commutation relations in an adapted frame, we see that
\[
X_i^*=-X_i+\sum_{k=1}^d \omega_{ik}^k,
\]
so that, 
\begin{align}\label{L}
L=\sum_{i=1}^d X_i^2 +X_0,
\end{align}
with
\begin{equation}\label{X0}
X_0=-\sum_{i,k=1}^d \omega_{ik}^k X_i.
\end{equation}

On sub-Riemannian manifolds with transverse symmetries, there is a canonical connection.

\begin{proposition}[See \cite{BG}]
There exists a unique connection $\nabla$ on $\mathbb{M}$ satisfying the following properties:
\begin{itemize}
\item[(i)] $\nabla g_\varepsilon =0$,  for all $\varepsilon >0$;
\item[(ii)] If $X$ and $Y$ are horizontal vector fields, $\nabla_{X} Y$ is horizontal;
\item[(iii)] If $Z \in \mathcal{V}$, $\nabla Z=0$;
\item[(iv)] If $X,Y$ are horizontal vector fields and $Z \in \mathcal{V}$,
the torsion vector field $T(X,Y)$ is vertical and $T(X, Z)=0$.
\end{itemize}
\end{proposition}

Intuitively $\nabla$ is the connection which coincides with  the Levi-Civita connection of the Riemannian metric $g_1$ on the horizontal bundle $\mathcal{H}$ and that parallelizes the Lie algebra $\mathcal{V}$.  We stress that this connection does not depend on $\varepsilon$ and straightforward computations show that one has in a local adapted frame:
\begin{equation}\label{christoffel1}
\nabla_{X_i} X_j= \sum_{k=1}^d
\frac{1}{2} \left(\omega_{ij}^k +\omega_{ki}^j +\omega_{kj}^i
\right) X_k,
\end{equation}
\begin{equation}\label{christoffel2}
\nabla_{Z_{m}} X_i=-\sum_{\ee=1}^d \delta_{im}^\ee X_\ee,
\end{equation}
\begin{equation}\label{christoffel3}
\nabla Z_{m}=0,
\end{equation}
and
\[
T(X_i,X_j)=-\sum_{m=1}^\di\gamma_{ij}^m Z_m.
\]
We observe that, thanks to \eqref{L} and \eqref{X0}, in a local adapted frame we have
\[
L=\sum_{i=1}^d X^2_i-\nabla_{X_i} X_i.
\]

To establish Bochner-Weitzenb\"ock formulas, it will expedient to work in normal frames.

\begin{lemma}
Let $x \in \mathbb{M}$. There exists a local  adapted frame of vector fields $$\{X_1,\cdots,X_d, Z_1, \cdots, Z_\di\}$$ around $x$, such that, at $x$, 
\[
\nabla_{X_i} X_j(x)=0.
\]
Such frame will be called an adapted normal frame around $x$.
\end{lemma}

\begin{proof}
Since $\nabla$ coincides with a Levi-Civita connection on the horizontal bundle, the result essentially boils down to the existence of normal frames in Riemannian geometry.
\end{proof}

Observe that in a normal adapted frame, we have $\omega_{ij}^k=0$ at the center of the frame. We now introduce some maps that will play an important role in the sequel. For $Z \in \mathcal{V}$, there is a  unique skew-symmetric map $J_Z$ defined on the horizontal bundle $\mathcal{H}$ such that for all horizontal vector fields $X$ and $Y$,
\begin{align}\label{Jmap}
g_\mathcal{H} (J_Z (X),Y)= g_\mathcal{V} (Z,T(X,Y)).
\end{align}
In a local adapted frame, we have
\[
J_{Z_m}(X_i)=-\sum_{j=1}^d \gamma^m_{ij} X_j .
\]
We then extend $J_{Z_m}$ to be 0 on the vertical bundle $\mathcal{V}$.

We finally recall the following definition that was introduced in \cite{BG}:

\begin{definition}
The sub-Riemannian manifold $\M$ is said to be of Yang-Mills type, if for every horizontal vector field $X$, and any adapted local frame $\{X_1,\cdots,X_d, Z_1, \cdots, Z_\di\}$
\[
 \sum_{\ee=1}^d(\nabla_{X_\ee} T) (X_\ee,X)=0.
\]
\end{definition}

A quick computation shows that  $\M$  is of Yang-Mills type if and only if for every $x \in \M$ and  any adapted normal frame $\{X_1,\cdots,X_d, Z_1, \cdots, Z_\di\}$ around $x$, we have at $x$,
\[
\sum_{i=1}^d X_i \gamma_{ij}^m=0, \quad 1 \le j \le d, 1 \le m \le \di.
\]
 
\

We conclude the section with simple examples of sub-Riemannian manifolds with transverse symmetries: The 3-dimensional model spaces in  $K$-contact geometry.

Given a number $\rho \in \R$, suppose that $\bG(\rho)$ is a simply connected three-dimensional Lie group whose Lie algebra $\mathfrak{g}$ has a
 basis $\left\{ X,Y ,Z \right\}$ satisfying:
\begin{itemize}
\item[(i)] $[X,Y]=Z$,
\item[(ii)] $[X,Z]= -\rho Y$,
\item[(iii)] $[Y,Z]=\rho X$.
\end{itemize} 
For instance, for $\rho=0$, $\bG(\rho)$ is the Heisenberg group. For $\rho=1$, $\bG(\rho)$ is $\mathbf{SU}(2)$ and for $\rho=-1$, $\bG(\rho)$ is $\mathbf{SL}(2)$. It is easy to see that if we consider the left-invariant distribution $\mathcal{H}$ generated by $\{ X,Y \}$ and chose for $g_{\mathcal{H}}$ the left-invariant metric that makes $\{X,Y\}$ orthonormal then $(\M, \mathcal{H}, g_{\mathcal{H}})$ is a Yang-Mills sub-Riemannian manifold with transverse symmetry $Z$.

The sub-Laplacian on  $\bG(\rho)$ is  the left-invariant, second-order differential operator
\begin{equation*}
L= X^{2}  + Y^{2}
\end{equation*}
and the connection $\nabla$ is given by
\[
\nabla_X Y=\nabla_Y X=\nabla_X Z=\nabla_Y Z=0
\]
and 
\[
\nabla_Z X=-\rho Y, \quad  \nabla_Z Y=\rho X.
\]

\section{Bochner-Weitzenb\"ock formulas for sub-Laplacians on one-forms}

The purpose of the section is to establish the Bochner-Weitzenb\"ock formula for the sub-Laplacian. This formula is the key to the stochastic representation of the heat semigroup on one-forms. The reader only interested in the probabilistic consequences of the formula may directly jump to Section 4 and admit Proposition \ref{gft} which summarizes the results proved in this section.

\

From now on, in all the paper we consider a Yang-Mills sub-Riemannian manifold $\M$ with transverse symmetries and adopt the notations of the previous section. In particular $L$ denotes the sub-Laplacian on $\M$. 

Obviously, there exist infinitely many second order differential operators $\mathcal{L}$ defined on one-forms such that for every smooth function $f$,
\[
 d L f=  \mathcal{L} d f,
\]
where $d$ is the exterior derivative. In Riemannian geometry, a canonical $\mathcal{L}$ that satisfies the above commutation is the Hodge-de Rham Laplacian. On sub-Riemannian manifolds, even contact manifolds, there is no such canonical sub-Laplacian (see \cite{Rumin}) on one-forms. However, in our case, we will see in this section that there is a distinguished one-parameter family of sub-Laplacians on one-forms which are optimal when interested in Bochner-Weitzenb\"ock's type formulas and that satisfy the above commutation.

\

We start with some general preliminaries about one-forms. By declaring a one-form horizontal (resp. vertical) if it vanishes on the vertical bundle $\mathcal{V}$ (resp. on the horizontal bundle $\mathcal{H}$), the splitting of the tangent space
 \[
 T_x \bM= \mathcal{H}(x) \oplus \mathcal{V}(x)
 \]
 gives a splitting of the cotangent space
  \[
 T^*_x \bM= \mathcal{H}^*(x) \oplus \mathcal{V}^*(x).
 \]

If $\{X_1,\cdots,X_d, Z_1, \cdots, Z_\di\}$ is a local  adapted frame, the dual frame will be denoted $\{\theta_1,\cdots,\theta_d, \nu_1, \cdots, \nu_\di\}$ and referred to as a local adapted coframe. With a slight abuse of notations, for $\varepsilon>0$, the metric on $ T^*_x \bM$ that makes $\{\theta_1,\cdots,\theta_d, \frac{1}{\sqrt{\varepsilon} } \nu_1, \cdots, \frac{1}{\sqrt{\varepsilon}} \nu_\di\}$ orthonormal will still be denoted $g_\varepsilon$ or $\langle \cdot, \cdot \rangle_\varepsilon$. This metric on the cotangent bundle can thus be written
 \begin{equation}\label{rdgh}
g_{\varepsilon}=g_\mathcal{H} \oplus \varepsilon  g_{\mathcal{V}}, \quad \varepsilon >0,
 \end{equation}
where $g_\mathcal{H}$ (resp. $g_\mathcal{V}$)  is the metric on $\mathcal{H}^*$ (resp. $\mathcal{V}^*$) that makes $\{\theta_1,\cdots,\theta_d\}$ (resp. $\{ \nu_1, \cdots, \nu_\di\}$ ) orthonormal. We use similar notations and conventions as before so that for every $\eta$ in $T^*_x \M$,
\[
\| \eta \|^2_{\varepsilon} =\| \eta \|_\mathcal{H}^2+\varepsilon \| \eta \|_\mathcal{V}^2.
\]

\

We will still denote by $L$ the covariant extension on one-forms of the sub-Laplacian. In a local adapted frame, we have thus
\[
L=\sum_{i=1}^d \nabla_{X_i}\nabla_{X_i} -\nabla_{\nabla_{X_i} X_i}.
\]

\

We define then $\mathfrak{Ric}_{\mathcal{H}}$ as the fiberwise symmetric linear map on one forms such that for every smooth functions $f,g$,
\[
\langle  \mathfrak{Ric}_{\mathcal{H}} (df), dg \rangle_\varepsilon=\mathbf{Ricci} (\nabla_\mathcal{H} f ,\nabla_\mathcal{H} g),
\]
where $\mathbf{Ricci}$ is the Ricci curvature of the connection $\nabla$. Of course, $ \mathfrak{Ric}_{\mathcal{H}} $ does not depend on $\varepsilon$ because the above definition implies that $ \mathfrak{Ric}_{\mathcal{H}} $ is horizontal, that is transforms any one-form into a horizontal form.  Actually, a computation shows that in a normal adapted frame around $x$, we have at $x$,
\[
\mathfrak{Ric}_{\mathcal{H}}(\eta)=\sum_{k,\ell=1}^d \frac{1}{2} ( \rho_{k\ell}+\rho_{\ell k}) f_k  \theta_{\ell}
\]
where $\eta=\sum_{i=1}^d f_i \theta_i +\sum_{m=1}^\di g_m \nu_m$ and
\[
\rho_{k\ell}= \sum_{j=1}^d
\sum_{m=1}^\di \gamma_{kj}^{m} \delta_{jm}^\ee +
\sum_{j=1}^d X_\ee\omega^j_{kj} - X_j\omega^k_{\ee j}.
\]

Finally, we consider the first order differential operator  $\mathfrak{J}$ defined in a local adapted frame by 
\[
\mathfrak{J} (\eta)= \sum_{i,j=1}^d \sum_{m=1}^\di \gamma_{ij}^m (X_j g_m) \theta_i, 
\]
where, again, $\eta=\sum_{i=1}^d f_i \theta_i +\sum_{m=1}^\di g_m \nu_m$. By defining $J_{Z_m}$ on one-forms using the duality
\[
J_{Z_m} (\theta_i) =J_{Z_m} (X_i), \quad J_{Z_m} (\nu_i) =0,
\]
 we can write more intrinsically
\[
\mathfrak{J}=\sum_{m=1}^\di J_{Z_m} (d \iota_{Z_m})
\]
where  $\iota$ is the interior product. This last expression shows that $\mathfrak{J}$  does not depend on the choice on the local frame, and is therefore a globally defined first order differential operator on one-forms.

We are now in a position to prove our first commutation result.

\begin{proposition}
Let 
\[
\square_\infty = L+2 \mathfrak{J} - \mathfrak{Ric}_{\mathcal{H}}.
\]
Then, we have for every smooth function $f$,
\begin{align}\label{commu}
dLf =\square_\infty df.
\end{align}
\end{proposition}

\begin{proof}
Let $x \in \M$. It is enough to prove this commutation at $x$  in a local adapted normal frame $\{ X_1, \cdots, X_d, Z_1, \cdots , Z_\di \}$ around $x$. Observing that $L$ and $Z_m$ commute (see \cite{BG}), we have:
\begin{align*}
dLf& =\sum_{i=1}^d (X_i L f) \theta_i +\sum_{m=1}^\di (Z_m L f )\nu_m \\
  &=\sum_{i=1}^d ( LX_i f) \theta_i +\sum_{m=1}^\di ( LZ_m f ) \nu_m+\sum_{i=1}^d [X_i,L] f \theta_i \\
  &=Ldf+\sum_{i=1}^d ([X_i,L] f) \theta_i .
\end{align*}
Keeping in mind that at the center of the frame $\omega_{ij}^k=0$, and thanks to the Yang-Mills assumption
\[
\sum_{i=1}^d X_i \gamma_{ij}^m=0,
\]
we now compute:
\begin{align*}
& \sum_{i=1}^d ([X_i,L] f) \theta_i \\
 =&\sum_{i,j=1}^d( [X_i,X_j^2] f )\theta_i+\sum_{i=1}^d ([X_i,X_0]f) \theta_i \\
=&\sum_{i=1}^d \left( [X_i,X_j]X_j f+X_j[X_i,X_j]f-\sum_{j,k=1}^d [X_i, \omega_{jk}^k X_j]f \right) \theta_i \\
=&\sum_{i=1}^d \left( \sum_{j=1}^d \sum_{m=1}^\di \gamma_{ij}^m (Z_mX_jf +X_jZ_mf )+\sum_{j,k=1}^d (X_j\omega_{ij}^k -X_i\omega_{jk}^k)X_kf \right)\theta_i \\
=&\sum_{i=1}^d \left( 2\sum_{j=1}^d \sum_{m=1}^\di \gamma_{ij}^m (X_jZ_mf )-\sum_{j,k=1}^d\sum_{m=1}^\di \gamma_{ij}^m \delta_{jm}^k X_kf +\sum_{j,k=1}^d (X_j\omega_{ij}^k -X_i\omega_{jk}^k)X_kf \right)\theta_i
\end{align*}
It is now elementary to identify the terms in the above equality.
\end{proof}

Obviously, $\square_\infty$ is not the only operator that satisfies \eqref{commu}. Actually, since $d^2=0$,  if $\Lambda$ is any fiberwise linear map from the space of two-forms into the space of one-forms, then we have
\[
dLf =(\square_\infty +\Lambda \circ d )df.
\]
This raises the question of an \textit{optimal} choice of $\Lambda$. The following proposition answers this question if optimality is understood in the sense of a corresponding  Bochner-Weitzenb\"ock's formula.

\begin{proposition}\label{optimal}
For any fiberwise linear map $\Lambda$ from the space of two-forms into the space of one-forms, and any $x \in \M$, we have
\begin{align*}
 & \inf_{\eta, \| \eta (x) \|_{\varepsilon}=1} \left(  \frac{1}{2} (L \| \eta \|_\varepsilon^2)(x) -\langle (\square_\infty +\Lambda \circ d )\eta (x) , \eta (x)\rangle_\varepsilon \right)  \\
 \le & \inf_{\eta, \| \eta (x) \|_{\varepsilon}=1} \left(  \frac{1}{2} (L \| \eta \|_\varepsilon^2)(x) -\left\langle \left(\square_\infty -\frac{1}{\varepsilon} T \circ d\right) \eta (x) , \eta (x) \right\rangle_\varepsilon \right),
\end{align*}
where in the above notation, the torsion tensor $T$ is interpreted, by duality, as a fiberwise linear map from the space of two-forms into the space of one-forms.
\end{proposition}

\begin{proof}
Let $x \in \M$ and consider a normal adapted frame around $x$. The following computations are done at the center $x$ of the frame. Let us consider a smooth one-form
\[
\eta=\sum_{i=1}^d f_i \theta_i+\sum_{m=1}^\di g_m \nu_m.
\]
 We have,
\begin{align}\label{gfdr}
  &  \frac{1}{2} (L \| \eta \|_\varepsilon^2) -\langle (\square_\infty +\Lambda \circ d )\eta  , \eta \rangle_\varepsilon \notag \\
  =&\sum_{i=1}^d \| \nabla_{\mathcal{H}} f_i \|_\mathcal{H}^2+\varepsilon \sum_{m=1}^\di  \| \nabla_{\mathcal{V}} g_m \|_\mathcal{V}^2-2\sum_{i,j=1}^d \sum_{m=1}^\di \gamma_{ij}^m (X_j g_m) f_i-\langle \Lambda (d \eta)  , \eta \rangle_\varepsilon +\langle \mathfrak{Ric}_\mathcal{H} \eta ,\eta \rangle_\mathcal{H}.
  \end{align}
  On the other hand, the exterior derivative can be computed as follows:
  \begin{align*}
  d\eta & =\sum_{i,j=1}^d\left(X_if_j-\frac{1}{2}\sum_{m=1}^\di \gamma_{ij}^m g_m \right)\theta_i \wedge \theta_j+\sum_{j=1}^d \sum_{m=1}^\di \left( X_j g_m-Z_mf_j -\sum_{i=1}^d \delta_{jm}^i f_i \right)\theta_j \wedge \nu_m \\
   &+\sum_{m,\ell=1}^\di \alpha_{m,\ell} \nu_\ell \wedge \nu_m,
  \end{align*}
  where $\alpha_{m,\ell}$ are coefficients which are unimportant to compute explicitly.
  Because of the vertical derivatives $Z_m f_i$ and $Z_\ell g_m$ that do not appear in \eqref{gfdr},  the quantity
  \begin{align}\label{plo}
  \inf_{\eta, \| \eta (x) \|_{\varepsilon}=1} \left(  \frac{1}{2} (L \| \eta \|_\varepsilon^2)(x) -\langle (\square_\infty +\Lambda \circ d )\eta (x) , \eta (x)\rangle_\varepsilon \right) 
  \end{align}
  is then  finite if and only if $\Lambda (\nu_\ell \wedge \nu_m)=\Lambda (\theta_i \wedge \nu_m)=0$, which we assume from now on. Also, clearly, every non zero term $\langle \Lambda (\theta_i \wedge \theta_j) , \theta_k \rangle_\mathcal{H}$ would decrease \eqref{plo}, so we can assume $\langle \Lambda (\theta_i \wedge \theta_j) , \theta_k \rangle_\mathcal{H}=0$. Completing the squares in \eqref{gfdr}, we see then that the quantity to be maximized is
  \[
  \inf_{\eta, \| \eta (x) \|_{\varepsilon}=1} \left( - \frac{1}{4} \varepsilon^2 \sum_{i,j=1}^d \left( \sum_{\ell=1}^\di g_\ell \langle \Lambda (\theta_i \wedge \theta_j) , \nu_\ell \rangle_\mathcal{V}\right)^2 +\frac{1}{2}  \varepsilon \sum_{i,j=1}^d\sum_{m,\ell=1}^\di \gamma_{ij}^m  g_m g_\ell  \langle \Lambda (\theta_i \wedge \theta_j) , \nu_\ell \rangle_\mathcal{V}  \right).
  \]
  We then easily see that the optimal choice of $ \langle \Lambda (\theta_i \wedge \theta_j) , \nu_\ell \rangle_\mathcal{V} $ is given by
  \[
  \langle \Lambda (\theta_i \wedge \theta_j) , \nu_\ell \rangle_\mathcal{V} =\frac{1}\varepsilon \gamma_{ij}^l.
  \]
\end{proof}

In the sequel, we shall denote
\[
\square_\varepsilon=\square_\infty -\frac{1}{\varepsilon} T \circ d.
\]
For our purpose, we will need to rewrite $\square_\varepsilon$ in a sum of squares form, from which we will be able to deduce a stochastic representation of the semigroup $e^{\frac{1}{2} t\square_\varepsilon }$. 

If $V$ is a horizontal vector field, we consider the fiberwise linear map from the space of one-forms into itself which is given by in a local adapted frame by
\[
\mathfrak{T}^\varepsilon_V \eta =-\sum_{j=1}^d \eta (T(V,X_j)) \theta_j +\frac{1}{2 \varepsilon} \sum_{m=1}^\di \eta( J_{Z_m} V) \nu_m.
\]
We see that $\mathfrak{T}^\varepsilon_V$ does not depend of the choice of the local adapted frame and thus, is a globally well-defined, smooth section. In a local adapted frame, if $\eta=\sum_{i=1}^d f_i \theta_i +\sum_{m=1}^\di g_m \nu_m$, then we have
\[
\mathfrak{T}^\varepsilon_{X_i} \eta =\sum_{j=1}^d \sum_{\ell=1}^\di \gamma_{ij}^\ell g_\ell \theta_j -\frac{1}{2\varepsilon}  \sum_{j=1}^d \sum_{m=1}^\di \gamma_{ij}^m f_j \nu_m.
\]
\begin{theorem}\label{sum}
In a local adapted frame, we have
\[
\square_\varepsilon=\sum_{i=1}^d (\nabla_{X_i} -\mathfrak{T}^\varepsilon_{X_i})^2 - ( \nabla_{\nabla_{X_i} X_i}-  \mathfrak{T}^\varepsilon_{\nabla_{X_i} X_i}) +\frac{1}{2 \varepsilon} \sum_{m=1}^\di J_{Z_m}^*J_{Z_m}- \mathfrak{Ric}_{\mathcal{H}},
\]
and for any smooth one-form $\eta$,
\[
\frac{1}{2} L \| \eta \|_{2\varepsilon}^2 -\langle \square_\varepsilon \eta , \eta \rangle_{2\varepsilon}=\sum_{i=1}^d  \| \nabla_{X_i} \eta  -\mathfrak{T}^\varepsilon_{X_i} \eta \|_{2\varepsilon}^2 +\mathfrak{Ric}_{\mathcal{H}}-\frac{1}{2 \varepsilon} \sum_{m=1}^\di J_{Z_m}^*J_{Z_m}.
\]
\end{theorem}

\begin{proof}
It is enough to prove the two identities at the center of an adapted normal frame. From the definition of $\square_\varepsilon$, at the center of the frame, we have for $\eta=\sum_{i=1}^d f_i \theta_i+\sum_{m=1}^\di g_m \nu_m$,
\begin{align*}
\square_\varepsilon&=\sum_{i=1}^d \nabla^2_{X_i}\eta +2 \sum_{i,j=1}^d \sum_{m=1}^\di \gamma_{ij}^m (X_jg_m) \theta_i+\frac{1}{\varepsilon} \sum_{i,j=1}^d \left( X_i f_j-\frac{1}{2} \sum_{m=1}^\di \gamma_{ij}^m g_m \right)\left( \sum_{m=1}^\di \gamma_{ij}^m \nu_m \right) \\
 &-\mathfrak{Ric}_{\mathcal{H}}.
\end{align*}
On the other hand, still at the center of the frame, we compute
\[
(\nabla_{X_i}-\mathfrak{T}^\varepsilon_{X_i}) \eta =\sum_{j=1}^d \left( X_i f_j -\sum_{\ell=1}^m \gamma_{ij}^\ell g_l \right)\theta_j +\sum_{m=1}^\di\left(X_i g_m +\frac{1}{2\varepsilon} \sum_{j=1}^d \gamma_{ij}^m f_j \right)\nu_m.
\]
Keeping in mind that in a local adapted frame, we have
\[
J_{Z_m}(X_i)=-\sum_{j=1}^d \gamma^m_{ij} X_j ,
\]
it is now an elementary exercise to check that
\[
\square_\varepsilon=\sum_{i=1}^d (\nabla_{X_i} -\mathfrak{T}^\varepsilon_{X_i})^2+\frac{1}{2 \varepsilon} \sum_{m=1}^\di J_{Z_m}^*J_{Z_m}- \mathfrak{Ric}_{\mathcal{H}}.
\]
The proof of the second identity follows the same lines as in the proof of Proposition \ref{optimal}. The details are let to the reader.
\end{proof}

If $V$ is a horizontal vector field, $\mathfrak{T}^\varepsilon_V$ is a skew-symmetric operator for the Riemannian metric $g_{2\varepsilon}$, as a consequence, $\square_\varepsilon$ is a symmetric operator for the metric $g_{2\varepsilon}$ on the space of smooth and compactly supported one-forms. It is interesting that $\square_\varepsilon$ is symmetric with respect to the metric $g_{2\varepsilon}$ but not $g_{\varepsilon}$ which is the one that was used to construct $\square_\varepsilon$.

\

The operator $\sum_{m=1}^\di J_{Z_m}^*J_{Z_m}$ does not depend on the choice of the frame and shall concisely be denoted by $J^*J$. We can note that in the case where $\M$ is a Sasakian manifold, like  the Heisenberg group for instance, $J^*J$ is the identity map on the horizontal  distribution. 

\

Similarly, the operator $\sum_{i=1}^d (\nabla_{X_i} -\mathfrak{T}^\varepsilon_{X_i})^2 - ( \nabla_{\nabla_{X_i} X_i}-  \mathfrak{T}^\varepsilon_{\nabla_{X_i} X_i})$ does not depend on the choice of the frame and can be more intrinsically described as follows.

If $\eta$ is a one-form, we define the horizontal gradient in a local adapted frame of $\eta$ as the $(0,2)$ tensor
\[
\nabla_\mathcal{H} \eta =\sum_{i=1}^d \nabla_{X_i} \eta \otimes \theta_i.
\]
Similarly, we will use the notation
\[
\mathfrak{T}^\varepsilon_\mathcal{H} \eta =\sum_{i=1}^d \mathfrak{T}^\varepsilon_{X_i} \eta  \otimes \theta_i.
\]
It is then easily seen that, in a local adapted frame,
\[
-(\nabla_\mathcal{H} -\mathfrak{T}_\mathcal{H}^\varepsilon)^* (\nabla_\mathcal{H} -\mathfrak{T}_\mathcal{H}^\varepsilon)=\sum_{i=1}^d (\nabla_{X_i} -\mathfrak{T}^\varepsilon_{X_i})^2 - ( \nabla_{\nabla_{X_i} X_i}-  \mathfrak{T}^\varepsilon_{\nabla_{X_i} X_i}),
\]
where the adjoint is of course understood with respect to the metric $g_{2\varepsilon}$.
 We therefore globally have
\[
\square_\varepsilon=-(\nabla_\mathcal{H} -\mathfrak{T}_\mathcal{H}^\varepsilon)^* (\nabla_\mathcal{H} -\mathfrak{T}_\mathcal{H}^\varepsilon)+\frac{1}{2 \varepsilon}J^* J - \mathfrak{Ric}_{\mathcal{H}}.
\]

To finish the section, we illustrate our formulas in the case of the model space $\mathbb{G}(\rho)$ that was introduced in Section 2. In that case, we have a basis of left invariant vector fields $\left\{ X,Y ,Z \right\}$ satisfying: $[X,Y]=Z$, $[X,Z]= -\rho Y$, and  $[Y,Z]=\rho X$ and the sub-Laplacian is given by 
\[
L=X^2+Y^2.
\]
Every one-form can be written as $\eta=f_1 \theta_1 +f_2 \theta_2 +g \nu$ where $\left\{ \theta_1,\theta_2 ,\nu \right\}$ is the dual basis of $\left\{ X,Y ,Z \right\}$. We identify $\eta$ with the column vector
\[
\eta=\left( 
\begin{array}{l}
f_1 \\
f_2 \\
g
\end{array}
\right)
\]
Elementary computations show then that
\[
 \mathfrak{Ric}_{\mathcal{H}}=
 \left(
 \begin{array}{lll}
 \rho & 0 & 0 \\
 0 & \rho &0 \\
 0 & 0 & 0
 \end{array}
 \right),
\]
\[
\square_\varepsilon
=
 \left(
 \begin{array}{lll}
 L- \rho & 0 & 2Y \\
 0 & L -\rho &-2X \\
 -\frac{1}{\varepsilon} Y  &  \frac{1}{\varepsilon} X  & L-\frac{1}{\varepsilon}
 \end{array}
 \right),
\]
\[
\mathfrak{T}_X
=
 \left(
 \begin{array}{lll}
 0 & 0 & 0 \\
 0 & 0 & 1 \\
 0&  -\frac{1}{2\varepsilon} & 0
 \end{array}
 \right)
\]
\[
\mathfrak{T}_Y
=
 \left(
 \begin{array}{lll}
 0 & 0 & -1 \\
 0 & 0 & 0 \\
\frac{1}{2\varepsilon} & 0 & 0
 \end{array}
 \right).
\]
and
\[
J^* J 
=
 \left(
 \begin{array}{lll}
 1 & 0 & 0 \\
 0 & 1 & 0 \\
0 & 0 & 0
 \end{array}
 \right).
\]
\section{Gradient formulas and  bounds for the heat semigroup}

Throughout the section, we work under the same assumptions as the previous section and we moreover assume  that for every horizontal one-form $\eta$,
\[
 \langle \mathfrak{Ric}_{\mathcal{H}} (\eta) , \eta  \rangle_\mathcal{H} \ge -K \| \eta \|^2_\mathcal{H} , \quad \langle J^*J \eta, \eta  \rangle_\mathcal{H} \le \kappa  \| \eta \|^2_\mathcal{H},
\]
with $K,\kappa \ge 0$. We also assume that the manifold $\M$ is metrically complete with respect to the sub-Riemannian distance. Under these assumptions, it was proved in \cite{BG} that the sub-Laplacian $L$ is essentially self-adjoint on $C_0^\infty(\M)$ and that the semigroup $P_t=e^{\frac{1}{2} tL} $ is stochastically complete.

The following result was proved in the previous section.

\begin{proposition}\label{gft}
Consider the operator defined on one-forms by the formula
 \[
\square_\varepsilon=-(\nabla_\mathcal{H} -\mathfrak{T}_\mathcal{H}^\varepsilon)^* (\nabla_\mathcal{H} -\mathfrak{T}_\mathcal{H}^\varepsilon)+\frac{1}{2 \varepsilon}J^* J - \mathfrak{Ric}_{\mathcal{H}},
\]
then for any smooth function $f$,
\[
dLf =\square_{\varepsilon} df
\]
and  for any smooth one-form $\eta$
\begin{align*}
\frac{1}{2} L \| \eta \|_{2\varepsilon}^2 -\langle \square_\varepsilon \eta , \eta \rangle_{2\varepsilon} & =\sum_{i=1}^d  \| \nabla_{X_i} \eta  -\mathfrak{T}^\varepsilon_{X_i} \eta \|_{2\varepsilon}^2 +\left\langle \left(\mathfrak{Ric}_{\mathcal{H}}-\frac{1}{2 \varepsilon} J^* J\right)\eta, \eta \right\rangle_{2\varepsilon} \\
 & \ge \left( \rho-\frac{\kappa}{2\varepsilon} \right) \| \eta \|^2_\mathcal{H} .
\end{align*}
\end{proposition}

\begin{remark}
We note again that the operator $\square_{\varepsilon}$ depends on $\varepsilon$, but since $dLf =\square_{\varepsilon} df$,  $\square_{\varepsilon} \eta$ does not depend on $\varepsilon$ when $\eta$ is an exact one-form.
\end{remark}

\subsection{Heat semigroup on one-forms}

We are interested in a stochastic representation of the semigroup on one-forms which is generated by $\square_\varepsilon$. This semigroup is well-defined by using the spectral theorem thanks to the following lemma.

\begin{lemma}
The operator  $\square_\varepsilon$ is essentially self-adjoint on the space of smooth and compactly supported one-forms for the Riemannian metric $g_{2\varepsilon}$.
\end{lemma}

\begin{proof}
Since we assume $\M$ to be metrically complete for the sub-Riemannian distance, it is also complete for the Riemannian distance associated to $g_{2\varepsilon}$, because $g_{2\varepsilon}$ is a Riemannian extension of $g_\mathcal{H}$. From \cite{strichartz1}, there exists therefore a sequence $h_n \in C^\infty_0(\M)$ , such that $0\le h_n \le 1$ and $\| \nabla_\mathcal{H} h_n \|^2_\infty +2\varepsilon \| \nabla_\mathcal{V} h_n \|_\infty^2 \to 0$. In particular, note that we have $\| \nabla_\mathcal{H} h_n \|_\infty \to 0$. 

\

To prove that $\square_\varepsilon$ is essentially self-adjoint it is enough (see \cite{strichartz1}) to prove that for some $\lambda >0$, $\square_\varepsilon \eta =\lambda \eta $ with $\eta \in L^2$ implies $\eta =0$. So, let $\lambda >0$ and $\eta \in  L^2$ such that $\square_\varepsilon \eta =\lambda \eta $. We have then
\begin{align*}
 & \lambda \int_\M h_n^2 \| \eta \|_{2\varepsilon}^2 \\
 =& \int_\M  \langle  h_n^2 \eta , \square_\varepsilon \eta \rangle_{2 \varepsilon} \\
 =&- \int_\M \langle \nabla_\mathcal{H} (h_n^2 \eta )-\mathfrak{T}_\mathcal{H}^\varepsilon (h_n^2 \eta) , \nabla_\mathcal{H} \eta -\mathfrak{T}_\mathcal{H}^\varepsilon \eta \rangle_{2\varepsilon} +\int_\M h_n^2 \left\langle \left(\frac{1}{2 \varepsilon}J^* J - \mathfrak{Ric}_{\mathcal{H}}\right)(\eta), \eta \right\rangle_{2 \varepsilon} \\
 =&-\int_\M h_n^2 \| \nabla_\mathcal{H} \eta -\mathfrak{T}_\mathcal{H}^\varepsilon \eta \|_{2\varepsilon}^2 -2\int_\M h_n  \langle \eta, \nabla_{\nabla_\mathcal{H} h_n} \eta \rangle_{2\varepsilon} +\int_\M h_n^2 \left\langle \left(\frac{1}{2 \varepsilon}J^* J - \mathfrak{Ric}_{\mathcal{H}}\right)(\eta), \eta \right\rangle_{2 \varepsilon} .
\end{align*}
From our assumptions, the symmetric tensor $\frac{1}{2 \varepsilon}J^* J - \mathfrak{Ric}_{\mathcal{H}}$ is bounded from above, thus by choosing $\lambda$ big enough, we have
\[
\int_\M h_n^2 \| \nabla \eta -\mathfrak{T}^\varepsilon \eta \|_{2\varepsilon}^2 +2\int_\M h_n  \langle \eta, \nabla_{\nabla_\mathcal{H} h_n} \eta \rangle_{2\varepsilon} \le 0.
\]
By letting $n\to \infty$, we easily deduce that $\| \nabla_\mathcal{H} \eta -\mathfrak{T}_\mathcal{H}^\varepsilon \eta \|_{2\varepsilon}^2=0$ which implies $ \nabla_\mathcal{H} \eta -\mathfrak{T}_\mathcal{H}^\varepsilon \eta=0$. If we come back to the equation $\square_\varepsilon \eta =\lambda \eta $ and the expression of $\square_\varepsilon$, we see that it implies that:
\[
\left( \frac{1}{2 \varepsilon}J^* J - \mathfrak{Ric}_{\mathcal{H}} \right) (\eta) =\lambda \eta.
\]
Our choice of $\lambda$ forces then $\eta=0$.
\end{proof}

Since $\frac{1}{2} \square_\varepsilon$ is essentially self-adjoint, it admits a unique self-adjoint extension which generates through the spectral theorem a semigroup $Q^\varepsilon_t=e^{\frac{1}{2} t \square_\varepsilon}$.  As already mentioned, we will denote by $P_t=e^{\frac{1}{2} t L}$ the semigroup generated by $\frac{1}{2} L$. We have the following commutation property:

\begin{lemma}\label{commu2}
If $f \in C_0^\infty(\M)$, then for every $t \ge 0$,
\[
d P_t f=Q^\varepsilon_t df.
\]
\end{lemma}
\begin{proof}
Let $\eta_t =Q^\varepsilon_t df$. It is the unique solution in $L^2$ of the heat equation
\[
\frac{\partial \eta}{\partial t}=\frac{1}{2}  \square_\varepsilon \eta,
\]
with initial condition $\eta_0 =df$. From \cite{BG}, we have that $\alpha_t=dP_t f$ is in $L^2$, and from the fact that 
\[
dL=\square_\varepsilon d,
\]
we see that $\alpha$ solves the heat equation
\[
\frac{\partial \alpha}{\partial t}=\frac{1}{2}  \square_\varepsilon \alpha
\]
with the same initial condition $\alpha_0=df$. We conclude thus $\alpha=\eta$.
\end{proof}

\subsection{Stochastic representation of the semigroup on one-forms}

We now turn to the stochastic representation of $Q_t^{\varepsilon}$. We denote by $(X_t)_{t \ge 0}$ the symmetric diffusion process generated by $\frac{1}{2} L$. Since $P_t$ is stochastically complete, $(X_t)_{t \ge 0}$ has an infinite lifetime.

Consider the process $\tau_t^\varepsilon:T^*_{X_t} \M \to T^*_{X_0} \M  $ to be the the solution of  the following covariant Stratonovitch stochastic differential equation:
\begin{align}\label{tau}
d \left[ \tau^\varepsilon_t \alpha (X_t)\right]= \tau^\varepsilon_t \left( \nabla_{\circ dX_t}  -\mathfrak{T}^\varepsilon_{\circ dX_t} + \frac{1}{2} \left( \frac{1}{ 2\varepsilon}J^* J -\mathfrak{Ric}_{\mathcal{H}} \right) dt \right)\alpha (X_t), \quad \tau^\varepsilon_0=\mathbf{Id},
\end{align}
where $\alpha$ is any smooth one-form. We have the following key estimate:

\begin{lemma}\label{estimtau}
For every $ t\ge 0$, we have almost surely,
\[
\| \tau^\varepsilon_t \alpha (X_t) \|_{2 \varepsilon} \le e^{\frac{1}{2}\left( K+\frac{\kappa}{2 \varepsilon} \right)t} \| \alpha (X_0) \|_{2 \varepsilon}.
\]
\end{lemma}
\begin{proof}
The estimate stems from the fact that $\mathfrak{T}^\varepsilon$ is skew-symmetric for the Riemannian metric $g_{2\varepsilon}$, which  implies that the connection $ \nabla -\mathfrak{T}^\varepsilon$ is metric. The deterministic upper bound on $\tau^\varepsilon$ is therefore a consequence of the pointwise lower bound on $\mathfrak{Ric}_{\mathcal{H}}- \frac{1}{ 2\varepsilon}J^* J $ and Gronwall's lemma. 

More precisely,  consider the process $\Theta_t^\varepsilon:T^*_{X_t} \M \to T^*_{X_0} \M  $ to be  the solution of  the following covariant Stratonovitch stochastic differential equation:
\begin{align}\label{tau}
d \left[ \Theta_t^\varepsilon \alpha (X_t)\right]= \Theta_t^\varepsilon \left( \nabla_{\circ dX_t}  -\mathfrak{T}^\varepsilon_{\circ dX_t} \right) \alpha (X_t), \quad \tau^\varepsilon_0=\mathbf{Id},
\end{align}
where $\alpha$ is any smooth one-form.  Since $\mathfrak{T}^\varepsilon$ is skew-symmetric, $\Theta_t^\varepsilon$ is an isometry for the Riemannian metric $g_{2\varepsilon}$. Consider now the multiplicative functional $(\mathcal{M}^\varepsilon_t)_{t \ge 0}$, solution of the equation
 \[
 \frac{d \mathcal{M}^\varepsilon_t}{dt}=\frac{1}{2} \mathcal{M}_t \Theta_t^\varepsilon  \left( \frac{1}{ 2\varepsilon}J^* J -\mathfrak{Ric}_{\mathcal{H}} \right)( \Theta_t^\varepsilon)^{-1}, \quad \mathcal{M}_0^\varepsilon=\mathbf{Id}.
 \]
 With the previous notations, we of course have $ \tau_t^\varepsilon= \mathcal{M}^\varepsilon_t \Theta_t^\varepsilon$. Thus, the upper bound on $ \tau^\varepsilon$ boils down to an upper bound on $ \mathcal{M}^\varepsilon$ which is obtained as a consequence of Gronwall's inequality.
\end{proof}


\begin{theorem}\label{repre1}
Let $\eta$ be a smooth and compactly supported one-form. Then for every $t \ge 0$, and $x \in \M$,
\[
(Q^\varepsilon_t \eta )(x)=\mathbb{E}_x \left(  \tau^\varepsilon_t \eta (X_t)\right).
\]
\end{theorem}

\begin{proof}
It is basically a consequence of the definition of $\tau_\varepsilon$ and  It\^o's formula which implies that for every $t \ge 0$ the process
\[
N_s=\tau_s^\varepsilon (Q^\varepsilon_{t-s} \eta) (X_s),
\]
is a martingale.
\end{proof}

Combining Lemma \ref{commu2} with Theorem \ref{repre1}, we get therefore the following representation for the derivative of the semigroup:

\begin{corollary}
Let $f \in C_0^\infty(\M)$. Then for every $t \ge 0$, and $x \in \M$,
\[
dP_t f (x)=\mathbb{E}_x \left(  \tau^\varepsilon_t df (X_t)\right).
\]
\end{corollary}
 
 This eventually leads to a neat gradient bound for the semigroup $P_t$.
 
 \begin{corollary}\label{GG}
 For every $f \in C_0^\infty(\M)$, $\varepsilon > 0$, $t \ge 0$,
\[
\sqrt{ \| \nabla_\mathcal{H}P_t f \|^2_{\mathcal{H}} + 2 \varepsilon \| \nabla_\mathcal{V}P_t f \|^2_{\mathcal{V}} } \le  e^{\frac{1}{2}\left(K+\frac{\kappa}{2\varepsilon} \right)t} P_t \left(\sqrt{ \| \nabla_\mathcal{H} f \|^2_{\mathcal{H}} +2\varepsilon \| \nabla_\mathcal{V} f \|^2_{\mathcal{V}} }\right).
\]
 \end{corollary}
 
We remark that this gradient bound is new in our framework and is stronger than similar gradient bounds in \cite{BB}. It also immediately implies that Hypothesis 1.4 of \cite{BG} is satisfied on Yang-Mills  sub-Riemannian manifolds with transverse symmetries.

 \subsection{Integration by parts formula}
 
 As before, we denote by $(X_t)_{t \ge 0}$ the $L$-diffusion process. The stochastic parallel transport for the connection $\nabla$ along the paths of $(X_t)_{t \ge 0}$ will be denoted by  $\para_{0,t}$. Since the connection $\nabla$ is horizontal,  the map $\para_{0,t}: T_{X_0} \M \to T_{X_t} \M$ is an isometry that preserves the horizontal bundle, that is, if $u \in \mathcal{H}_{X_0}$, then $\para_{0,t} u \in \mathcal{H}_{X_t}$.  We see then that the anti-development of $(X_t)_{t \ge 0}$,
 \[
 B_t=\int_0^t \para_{0,s}^{-1} \circ dX_s,
 \]
 is a Brownian motion in the horizontal space $\mathcal{H}_{X_0}$. The following integration by parts formula will play an important role in the sequel.
 
 \begin{proposition}\label{IPP}
Let $x \in \M$.  For any $C^1$ adapted process $\gamma:\mathbb{R}_{\ge 0} \to \mathcal{H}_{x} $ such that $\mathbb{E}_x\left(\int_0^{+\infty} \| \gamma'(s) \|_\mathcal{H}^2 ds\right)<+\infty$  and any $f \in C^\infty_0(\M)$, $t \ge 0$,
\[
\mathbb{E}_x \left( f(X_t) \int_0^t \langle \gamma'(s),dB_s\rangle_{\mathcal{H}}  \right)=\mathbb{E}_x \left(\left\langle  \tau^\varepsilon_t df (X_t) ,\int_0^t  (\tau^{\varepsilon,*}_s)^{-1}  \para_{0,s} \gamma'(s) ds  \right\rangle_{2\varepsilon} \right).
\]
 \end{proposition}
 
 \begin{proof}
  We fix $t \ge 0$ and denote
\[
N_s=\tau_s^\varepsilon (dP_{t-s} f) (X_s).
\]
It is a martingale process. We have then for $f \in C_0^\infty(\M)$,
\begin{align*}
\mathbb{E}_x \left( f(X_t) \int_0^t \langle \gamma'(s),dB_s\rangle_{\mathcal{H}}  \right)&=\mathbb{E}_x \left( f(X_t) \int_0^t \langle \para_{0,s} \gamma'(s),\para_{0,s}dB_s\rangle_{\mathcal{H}}  \right) \\
 &=\mathbb{E}_x \left( ( f(X_t) -\mathbb{E}_x \left( f(X_t)\right)) \int_0^t \langle \para_{0,s} \gamma'(s),\para_{0,s}dB_s\rangle_{\mathcal{H}}  \right) \\
 &=\mathbb{E}_x \left(\int_0^t \langle dP_{t-s}f (X_s), \para_{0,s} dB_s \rangle_{\mathcal{H}}  \int_0^t \langle \para_{0,s} \gamma'(s),\para_{0,s}dB_s\rangle_{\mathcal{H}}  \right) \\
 & =\mathbb{E}_x \left(\int_0^t \langle dP_{t-s}f (X_s), \para_{0,s} \gamma'(s) \rangle_{\mathcal{H}} ds  \right) \\
 &=\mathbb{E}_x \left(\int_0^t \langle \tau_s^\varepsilon dP_{t-s}f (X_s), (\tau^{\varepsilon,*}_s)^{-1}  \para_{0,s} \gamma'(s) \rangle_{2\varepsilon} ds  \right)  \\
 &=\mathbb{E}_x \left(\int_0^t \langle N_s , (\tau^{\varepsilon,*}_s)^{-1}  \para_{0,s} \gamma'(s) \rangle_{2\varepsilon} ds  \right) \\
 &=\mathbb{E}_x \left(\left\langle  N_t ,\int_0^t  (\tau^{\varepsilon,*}_s)^{-1}  \para_{0,s} \gamma'(s) ds  \right\rangle_{2\varepsilon} \right),
\end{align*}
where we integrated by parts in the last equality.
 \end{proof}
 
 Let us observe that we can reinterpret the integration by parts formula of Proposition \ref{IPP} in a slightly different way.
 
  \begin{corollary}\label{IPP2}
Let $x \in \M$.  For any $C^1$ adapted process $\gamma:\mathbb{R}_{\ge 0} \to \mathcal{H}_{x} $ such that $\mathbb{E}_x\left(\int_0^{+\infty} \| \gamma'(s) \|_\mathcal{H}^2 ds\right)<+\infty$  and any $f \in C^\infty_0(\M)$, $t \ge 0$,
\[
\mathbb{E}_x\left( \left\langle df(X_t), \para_{0,t} v(t) \right\rangle_{2\varepsilon} \right)=\mathbb{E}_x \left( f(X_t) \int_0^t \langle \gamma'(s),dB_s\rangle_{\mathcal{H}}  \right),
\]
where $v$ is the solution of the Stratonovitch stochastic differential equation in $T_x \M$,:
\begin{align}\label{controlSDE}
\begin{cases}
dv(t)&=\para_{0,t}^{-1} \left(\mathfrak{T}^\varepsilon_{\circ dX_t} + \frac{1}{2} \left( \frac{1}{ 2\varepsilon}J^* J -\mathfrak{Ric}_{\mathcal{H}} \right) dt  \right) \para_{0,t} v(t) +\gamma'(t)dt \\
v(0)&=0. \notag
\end{cases}
\end{align}
 \end{corollary}
 \begin{proof}
 It is a consequence of It\^o's formula that 
 \[
 v(t)=\para_{0,t}^{-1} \tau_t^{\varepsilon,*} \int_0^t  (\tau_s^{\varepsilon,*} )^{-1} \para_{0,s} \gamma'(s) ds
 \]
 is the solution of the above stochastic differential equation. We conclude then with Proposition \ref{IPP}.
 \end{proof}

 As an immediate consequence of the integration by parts formula, we obtain the following Clark-Ocone type representation.

 \begin{proposition}\label{CO}
 Let $X_0=x \in \M$. For every $f \in C_0^\infty(\M)$, and every $t \ge 0$,
 \[
 f(X_t)=P_tf(x) +\int_0^t \left\langle \mathbb{E}_x \left( (\tau^\varepsilon_s)^{-1} \tau^\varepsilon_t df(X_t) \mid \mathcal{F}_s \right), \para_{0,s} dB_s \right \rangle_{\mathcal{H}} ,
 \]
 where $(\mathcal{F}_t)_{t \ge 0}$ is the natural filtration of $(B_t)_{t \ge 0}$.
 \end{proposition}
 
 \begin{proof}
Let $t \ge 0$. From It\^o's integral representation theorem, we can write
 \[
 f(X_t)=P_tf(x) +\int_0^t \left\langle a_s,  dB_s \right \rangle_{\mathcal{H}} ,
 \]
 for some adapted and square integrable $(a_s)_{0 \le s \le t}$.  Using the Proposition \ref{IPP}, we obtain therefore,
 \[
\mathbb{E}_x \left(  \int_0^t \langle \gamma'(s),a_s\rangle_{\mathcal{H}}  ds \right)=\mathbb{E}_x \left(\left\langle  \tau^\varepsilon_t df (X_t) ,\int_0^t  (\tau^{\varepsilon,*}_s)^{-1}  \para_{0,s} \gamma'(s) ds  \right\rangle_{2\varepsilon} \right).
\]
Since $\gamma'$ is arbitrary, we obtain that
\[
a_s= \mathbb{E}_x \left(  \para^{-1}_{0,s}  (\tau^\varepsilon_s)^{-1} \tau^\varepsilon_t df(X_t) \mid \mathcal{F}_s \right).
\]
 \end{proof}
 
 We deduce first the following Poincar\'e inequality for the heat kernel measure.
 
 \begin{proposition}
 For every $f \in C_0^\infty(\M)$, $t \ge 0$, $x \in \M$, $\varepsilon >0$,
 \[
 P_t(f^2)(x) -(P_t f)^2(x) \le \frac{ e^{\left(K+\frac{\kappa}{2\varepsilon} \right)t} -1}{K+\frac{\kappa}{2\varepsilon}} \left[ P_t (\| \nabla_\mathcal{H} f \|^2)(x) + 2\varepsilon P_t (\| \nabla_\mathcal{V} f \|^2)(x) \right]
 \]
 \end{proposition}
 
 \begin{proof}
 From the previous proposition and  Lemma \ref{estimtau} we have
 \[
 \mathbb{E}_x\left((f(X_t)-P_tf(x))^2 \right) \le \int_0^t e^{\left( K+\frac{\kappa}{2 \varepsilon} \right)(t-s)}  ds P_t (\| df \|_{2\varepsilon}^2)(x).
 \]
 \end{proof}
 
 We also get the log-Sobolev inequality for the heat kernel measure.
 
 \begin{proposition}
 For every $f \in C_0^\infty(\M)$, $t \ge 0$, $x \in \M$, $\varepsilon >0$,
 \[
 P_t(f^2\ln f^2 )(x) -P_t (f^2)(x)\ln P_t (f^2)(x)  \le 2 \frac{ e^{\left(K+\frac{\kappa}{2\varepsilon} \right)t} -1}{K+\frac{\kappa}{2\varepsilon}} \left[ P_t (\| \nabla_\mathcal{H} f \|^2)(x) +2 \varepsilon P_t (\| \nabla_\mathcal{V} f \|^2)(x) \right]
 \]
 \end{proposition}
 
 \begin{proof}
 The method for proving the log-Sobolev inequality from a representation theorem like Proposition \ref{CO} is due to \cite{capitaine} and the argument is easy to reproduce in our setting. Denote $G=f(X_t)^2$ and consider the martingale $N_s= \mathbb{E}( G | \mathcal{F}_s)$. Applying now It\^o's formula to $N_s \ln N_s$ and taking expectation yields
 \[
 \mathbb{E}_x( N_t \ln N_t)-\mathbb{E}_x( N_0 \ln N_0)=\frac{1}{2} \mathbb{E}_x\left( \int_0^t \frac{d[N]_s}{N_s} \right),
 \]
 where $[N]$ is the quadratic variation of $N$. From Proposition \ref{CO} applied with $f^2$, we have
 \[
dN_s=2 \left\langle \mathbb{E} \left( f(X_t)(\tau^\varepsilon_s)^{-1} \tau^\varepsilon_t df(X_t) \mid \mathcal{F}_s \right), \para_{0,s} dB_s \right \rangle_{\mathcal{H}}.
 \]
 Thus we have from Cauchy-Schwarz inequality
  \begin{align*}
 \mathbb{E}_x( N_t \ln N_t)-\mathbb{E}_x( N_0 \ln N_0) & \le 2 \mathbb{E}_x\left( \int_0^t \frac{\|\mathbb{E} \left( f(X_t)(\tau^\varepsilon_s)^{-1} \tau^\varepsilon_t df(X_t) \mid \mathcal{F}_s \right)\|_{2\varepsilon}^2}{N_s} ds \right) \\
  & \le 2 \int_0^t e^{\left( K+\frac{\kappa}{2 \varepsilon} \right)(t-s)} dsP_t (\| df \|_{2\varepsilon}^2)(x).
 \end{align*}
 \end{proof}

\subsection{Positive curvature and convergence to equilibrium}

In this final section we prove that if the tensor  $ \mathfrak{Ric}_{\mathcal{H}}$ is bounded from below by a positive constant on the horizontal bundle, then by exploiting a further geometric quantity we can prove convergence of the semigroup when $t \to +\infty$ and get sharp quantitative estimates in the form of a Poincar\'e and a log-Sobolev inequality with an exponential decay for the heat kernel measure.

So, we assume  throughout the section that for every horizontal one-form $\eta$,
\[
 \langle \mathfrak{Ric}_{\mathcal{H}} (\eta) , \eta  \rangle_\mathcal{H} \ge \rho_1 \| \eta \|^2_\mathcal{H} , \quad \langle J^*J \eta, \eta  \rangle_\mathcal{H} \le \kappa  \| \eta \|^2_\mathcal{H},,
\]
and that for every vertical one-form $\eta$, and any horizontal coframe $\{ \theta_1,\cdots,\theta_d \}$,
\[
\frac{1}{4} \sum_{\ell,j=1}^d \langle T(\theta_\ell, \theta_j), \eta \rangle^2_{\mathcal{V}} \ge \rho_2 \| \eta \|_\mathcal{V}^2,
\]
where $\rho_1,\rho_2 >0$ and $\kappa \ge 0$.
As proved in Section 3, this implies that for  for every one-form $\eta$, 
\[
\frac{1}{2} L \| \eta \|_\varepsilon^2-\langle \square_\varepsilon  \eta  , \eta \rangle_\varepsilon \ge \left( \rho_1-\frac{\kappa}{\varepsilon} \right) \| \eta_\mathcal{H} \|^2_\mathcal{H} +\rho_2  \| \eta_\mathcal{V} \|_\mathcal{V}^2.
\]
This implies of course
\[
\frac{1}{2} L \| \eta \|_\varepsilon^2-\langle \square_\varepsilon  \eta  , \eta \rangle_\varepsilon 
\ge \inf  \left( \rho_1-\frac{\kappa}{\varepsilon}, \frac{\rho_2}{\varepsilon}  \right)  \| \eta \|^2_{\varepsilon}.
\]
The constant  $\inf  \left( \rho_1-\frac{\kappa}{\varepsilon}, \frac{\rho_2}{\varepsilon}  \right) $ is maximal when $\rho_1-\frac{\kappa}{\varepsilon}= \frac{\rho_2}{\varepsilon} $, that is $\varepsilon= \frac{\kappa+\rho_2}{\rho_1}$. For this choice of $\varepsilon$, we have then 
\[
\inf  \left( \rho_1-\frac{\kappa}{\varepsilon}, \frac{\rho_2}{\varepsilon}  \right) =\frac{\rho_1 \rho_2}{\kappa +\rho_2}.
\]

We have then following estimate which is obtained by analyzing the It\^o-Stratonovitch correction term in the stochastic differential equation \eqref{tau}.

\begin{lemma}
Let  $\varepsilon=\frac{\kappa+\rho_2}{\rho_1}$. For every $ t\ge 0$,
\[
\mathbb{E} \left( \| \tau^\varepsilon_t \alpha (X_t) \|^2_{ \varepsilon}\right) \le e^{-\frac{\rho_1 \rho_2}{\kappa +\rho_2} t} \mathbb{E} \left( \| \alpha (X_0) \|^2_{ \varepsilon}\right).
\]
\end{lemma}

Arguing then as before, we obtain the following Bakry-\'Emery, Poincar\'e and log-Sobolev inequalities.

 \begin{proposition}
 For every $f \in C_0^\infty(\M)$, $t \ge 0$, $x \in \M$,
 \[
 \| \nabla_{\mathcal{H}}  P_t f \|^2+\frac{\kappa+\rho_2}{\rho_1}  \| \nabla_{\mathcal{V}}  P_t f \|^2 \le e^{-\frac{\rho_1 \rho_2}{\kappa +\rho_2} t} \left[ P_t (\| \nabla_\mathcal{H} f \|^2)(x) +  \frac{\kappa+\rho_2}{\rho_1} P_t (\| \nabla_\mathcal{V} f \|^2)(x) \right]
 \]
 \[
P_t(f^2)(x) -(P_t f)^2(x) \le  \frac{\kappa+\rho_2}{\rho_1 \rho_2}  \left( 1-e^{-\frac{\rho_1 \rho_2}{\kappa +\rho_2} t} \right)  \left[ P_t (\| \nabla_\mathcal{H} f \|^2)(x) + \frac{\kappa+\rho_2}{\rho_1}  P_t (\| \nabla_\mathcal{V} f \|^2)(x) \right],
 \]
 and
\begin{align*}
 &  P_t(f^2\ln f^2 )(x) -P_t (f^2)(x)\ln P_t (f^2)(x)  \\
 \le & 2 \frac{\kappa+\rho_2}{\rho_1 \rho_2}  \left( 1-e^{-\frac{\rho_1 \rho_2}{\kappa +\rho_2} t} \right)  \left[ P_t (\| \nabla_\mathcal{H} f \|^2)(x) + \frac{\kappa+\rho_2}{\rho_1}  P_t (\| \nabla_\mathcal{V} f \|^2)(x) \right].
 \end{align*}

 \end{proposition}
 
The first of the above inequality was already proved in \cite{BB} by completely different methods  and implies $\mu(\M)<+\infty$ and also  that when $t \to +\infty$, in $L^2$, $P_t f \to \frac{1}{\mu(\M)}$. It is worth pointing out that in the present framework, the two above Poincar\'e and log-Sobolev  inequalities are new but, by taking the limit when $t \to \infty$ we get the two inequalities
 \[
 \int_\M f^2 d\mu -\left( \int_\M f d\mu \right)^2 \le \frac{\kappa+\rho_2}{\rho_1 \rho_2}   \left[\int_\M \| \nabla_\mathcal{H} f \|^2  d\mu  +  \frac{\kappa+\rho_2}{\rho_1} \int_\M \| \nabla_\mathcal{V} f \|^2 d\mu \right]
 \]
and
 \begin{align*}
  \int_\M f^2\ln f^2 d\mu  -\int_\M f^2 d\mu \ln \int_\M f^2 d\mu \le  \frac{2(\kappa+\rho_2)}{\rho_1 \rho_2}   \left[\int_\M \| \nabla_\mathcal{H} f \|^2  d\mu  +  \frac{\kappa+\rho_2}{\rho_1} \int_\M \| \nabla_\mathcal{V} f \|^2 d\mu \right].
 \end{align*}
which were also already proved in \cite{BB} with the very same constants.

\end{document}